\documentclass[12pt,reqno]{amsart}
\usepackage{amssymb,amsmath,amsthm}
\usepackage{a4}
\usepackage{enumitem}
\usepackage{color}
\usepackage[colorlinks=true,
linkcolor=webgreen,
filecolor=webbrown,
citecolor=webgreen]{hyperref}
\definecolor{webgreen}{rgb}{0,0,1}
\definecolor{recrown}{rgb}{1,.2,.6}
\usepackage{fullpage}
\usepackage{orcidlink}

\begin{document}
\newtheorem{theorem}{Theorem}
\newtheorem{corollary}[theorem]{Corollary}
\newtheorem{lemma}[theorem]{Lemma}
\theoremstyle{definition}
\newtheorem{example}{Example}
\newtheorem*{examples}{Examples}
\newtheorem*{notation}{Notation}
\theoremstyle{theorem}
\newtheorem{thmx}{Theorem}
\renewcommand{\thethmx}{\text{\Alph{thmx}}}

\theoremstyle{theorem}
\newtheorem{corolmx}{Corollary}
\renewcommand{\thecorolmx}{\text{\Alph{corolmx}}}

\newtheorem{lemmax}{Lemma}
\renewcommand{\thelemmax}{\text{\Alph{lemmax}}}
\hoffset=-0cm
\theoremstyle{definition}
\newtheorem*{definition}{Definition}
\newtheorem*{remark}{\bf Remark}

\title{\bf Some factorization results for formal power series}
\author{Rishu Garg$^{1}$ {\large \orcidlink{0009-0008-2348-8340}}}
\author{Jitender Singh$^{2,\dagger}$ {\large \orcidlink{0000-0003-3706-8239}}}
\address[1,2]{Department of Mathematics,
Guru Nanak Dev University, Amritsar-143005, India\newline 
{\tt rishugarg128@gmail.com}, {\tt jitender.math@gndu.ac.in}}
\markright{}
\date{}
\footnotetext[2]{Corresponding author email(s): {\tt jitender.math@gndu.ac.in}

\url{https://sites.google.com/view/sonumaths3/home}\\

2020MSC: {Primary 13F25; 11Y05; 13P05}\\

\emph{Keywords}: Formal power series; Irreducibility; Integer coefficients; Factorization; Dumas irreducibility criterion; Newton polygon.
}
\maketitle
\newcommand{\K}{\mathbb{K}}
\begin{abstract}
In this paper, we obtain some factorization results on formal power series over principle ideal domains with sharp bounds on number of irreducible factors. These factorization results correspondingly lead to irreducibility criteria for formal power series. The information about prime factorization of the constant term up to a unit and that of some higher order terms is utilized for the purpose. Further, using theory of Newton polygons for power series, we extend the classical Dumas irreducibility criterion to formal power series over discrete valuation domains, which  in particular, yields several irreducibility criteria.
\end{abstract}
\section{Introduction}
Power series arise in several areas of mathematics such as in analysis, combinatorics, and algebra. In analysis, convergence of power series is of main interest. On the other hand, in combinatorics and algebra the algebraic properties of power series are of main interest. For a refreshing survey on algebraic theory of formal power series and applications, the reader may refer to  the comprehensive reviews in \cite{Niven,Sambale}.

Throughout, $R$ will denote a commutative ring with 1, and $R[[z]]$ the ring of formal power series over $R$ in the indeterminate $z$. A power series $f\in R[[z]]$  is invertible if and only if so does $f(0)$ in $R$. So, the question of factorization of $f$ in the ring $R[[z]]$ makes sense if $f(0)$ is not a unit in $R$. In view of this, the rings of our interest are those rings $R$ for which  the power series ring $R[[z]]$ is a unique factorization domain (UFD). If $R$ is a principal ideal domain (PID), then $R[[z]]$ is a UFD. However, many pathologies are known to exist in the power series ring $R[[z]]$ in comparison to the corresponding polynomial ring $R[z]$. In view of this, if $R$  is a UFD, then $R[[z]]$ need not be a UFD (See \cite{Samuel1961,Brewer}). Nevertheless, one can relate the ideal structure of the base ring with that of its power series ring. This can sometime be very useful in understanding factorization of power series over PIDs. We mention here that the power series rings have been the object of research in many contexts such as the dimension theory \cite{Arnold1,Arnold2,Park1,Park2}, the valuation overrings, the quotient field, the integral dependence, the cancelation problem, the valuation domains, the root extension, the study of subrings, the factorization properties etc. (see for the detail \cite{Gilmer,Brewer,Park3,Toan,Park4,Chang21,Gilmer2,A24,Hizem24} and references cited therein).

One of the most important ring of formal power series is $\mathbb{Z}[[z]]$, which is a UFD. The irreducibility properties in the polynomial ring $\Bbb{Z}[z]$ and the ring of formal power series $\Bbb{Z}[[z]]$  are generally not correlated. For example, for distinct prime numbers  $p$ and $q$, the polynomial $pq+z$ is irreducible in $\Bbb{Z}[z]$ but can be factored into two irreducible elements in $\Bbb{Z}[[z]]$, although here, the explicit factorization is not straightforward. Likewise, the polynomial $p+(1+pq)z+qz^2$ is irreducible in $\Bbb{Z}[[z]]$ but factors as $(p+z)(1+qz)$ in $\Bbb{Z}[z]$. In fact some factorization properties of the ring $\Bbb{Z}[[z]]$ have been obtained in the paper \cite{Daniel2008}, wherein numerous irreducibility criteria are also witnessed. On these lines, some results on irreducibility and factorization of power series in the ring $\Bbb{Z}[[z]]$ have been obtained recently in \cite{Dutta2021,Dutta2024}.

More generally, factorization properties and irreducibility of polynomials and formal power series over a PID $R$  in the ring $R[[z]]$ have been obtained in the paper  \cite{Elliot}. In view of this, we have the following well known factorization result  \cite{Daniel2008,Elliot}.
\begin{thmx}\label{th:A}
Let $R$ be a principle ideal domain. Let $f=\sum_{i=0}^\infty a_iz^i\in R[[z]]$ be such that $a_0$ is a nonzero nonunit. If $a_0$ is a product of two nonassociate elements  $m$ and $n$ in $R$, then $f$ is not irreducible in $R[[z]]$.
\end{thmx}
We remark that the conclusion of Theorem \ref{th:A} does not hold incase $R$ is not a PID and $R[[z]]$ is still a UFD (See \cite[Corollary 2.3(3)]{Elliot}).
\begin{proof}[\bf Proof of Theorem \ref{th:A}] Since $R$ is a PID and $m$ and $n$ are not associates, the ideal generated by $m$ and $n$ is $R$. Consequently, there exist $u$ and $v$ in $R$ for which  $mu+nv=1$. Now let $g=\sum_{i=0}^\infty b_i z^i \in R[[z]]$ be such that the coefficients of $g$ are defined recursively by
\begin{eqnarray*}
b_0=0;~b_i=a_i-uv\sum_{t=1}^{i}b_tb_{i-t},~i=1,2,\ldots.
\end{eqnarray*}
Then $f$ admits a factorization $f(z)=(m+vg(z))(n+ug(z))$ in $R[[z]]$.
\end{proof}
Theorem \ref{th:A} tells us that the irreducibility in the UFD $R[[z]]$ needs investigation only for those power series in which the constant term is an associate to a  prime power. In view of this, if $f=\sum_{i=0}^\infty a_i z^i \in \Bbb{Z}[[z]]$ is such that $|a_0|$ is prime,  then $f$ is irreducible in $\Bbb{Z}[[z]]$. 
This result follows from the following irreducibility criterion proved in \cite{Elliot}.
\begin{thmx}\label{th:B}
   Let $R$ be a principal ideal domain. Let $f=\sum_{i=0}^\infty a_i z^i \in R[[z]]$ be such that  $a_0$ is an associate to $p^k$ for some prime $p$ in $R$ and a positive integer $k$. If  $k=1$, or $p$ does not divide $a_1$, then $f$ is irreducible in $R[[z]]$.
\end{thmx}
As an application of Theorems \ref{th:A} and \ref{th:B}, we may have the following factorization result.
\begin{thmx}\label{th:1b}
Let $R$ be a principal ideal domain. Let $f = \sum_{i=0}^\infty a_i z^i\in R[[z]]$ be such that $a_0=u\prod_{i=1}^r p_i^{k_i}$ where $r\geq 1$, $u$ is a unit in $R$; $p_1,\ldots,p_r$ are pairwise nonassociate primes in $R$ and $k_1,\ldots,k_r$ are all positive integers. If $a_0$ and $a_1$ are nonassociates, then $f$ is a product of exactly $r$ irreducible factors in $R[[z]]$.
\end{thmx}
\begin{proof}[\bf Proof of Theorem \ref{th:1b}] The case $r=1$ reduces to one of the parts of Theorem \ref{th:B}. So, we may assume that $r>1$. Let $u=u_1\cdots u_r$ where $u_1,\ldots,u_r$ are units in $R$. By following the proof of Theorem \ref{th:A}, we find that $f$ factors as
\begin{eqnarray*}
f(z)=(u_1p_1^{k_1}+f_1(z))\cdots (u_rp_r^{k_r}+f_r(z)),
\end{eqnarray*}
where  $f_i=\sum_{j=0}^\infty b_{ij}z^j\in R[[z]]$ with $b_{i0}=0$  for $i=1,\ldots r$. Then we have
\begin{eqnarray*}
a_1=\sum_{i=1}^r\Bigl(b_{i1}uu_i^{-1}\prod_{j\neq i}{p_j^{k_j}}\Bigr)\equiv b_{t1}uu_t^{-1}\prod_{j\neq t}{p_j^{k_j}}\mod p_t,~t=1,\ldots, r,
\end{eqnarray*}
which in view of the fact that $a_0$ and $a_1$ are nonassociates tells us that $p_t$ does not divide $b_{t1}$ for every $t=1,\ldots, r$. By Theorem \ref{th:B}, $g_t=u_tp_t^{k_t}+f_t$ is irreducible in  $R[[z]]$ for each $t=1,\ldots, r$, and so, $f$ is the product of $r$ irreducible factors $g_1,\ldots,g_r$ in $R[[z]]$.
\end{proof}
In \cite{Daniel2007}, the authors characterized irreducibility of quadratic polynomials in $\Bbb{Z}[[z]]$, and reducibility of quadratic polynomials having integer coefficient in $\Bbb{Z}[[z]]$ was shown to be equivalent to their reducibility in the ring $\mathbb{Z}_p[z]$, the ring of polynomials over the $p$-adic integers. On these lines, the authors in \cite{Daniel2007} proved the following reducibility result for a class of higher degree polynomials.
\begin{thmx}\label{th:A2}
Let $f=\sum_{i=0}^d a_iz^i\in \Bbb{Z}[z]$ be such that there exists a prime $p$ and integers $n\geq 2$, $m\geq 1$, and $\gamma\geq 1$,  such that  $a_0=p^n$,  $a_1=p^m \gamma$, where $p\nmid \gamma$, $p\nmid a_i$ for all $i\geq 2$.
\begin{enumerate}
    \item[(i)] If $n>2m$, then $f$ is reducible in  $\Bbb{Z}[[z]]$ as well as $\mathbb{Z}_p[z]$.
    \item[(ii)] If $n\leq 2m$, then $f$ is reducible in $\Bbb{Z}[z]$ if and only if $f$ has a root  $\rho\in\mathbb{Z}_p$ with $v_p(\rho)\geq 1$.
\end{enumerate}
\end{thmx}
Recently using Newton polyhedron approach the authors in \cite{Rond} proved an irreducibility criterion for testing irreducibility of polynomials in one variable over the power series ring $\K[[x_1,\ldots,x_n]]$ in $n$ variables over a field $\K$.

It is worthwhile to mention that the classical irreducibility criteria for testing irreducibility of polynomials having integer coefficients due to Sch\"onemann, Eisenstein, Dumas, Perron and others  are available in the literature (see \cite{SKJS2023} for a survey on this topic). However, devising of irreducibility criteria for formal power series over integers has remained largely missing in the past literature. This lack of availability  of irreducibility criteria in abundance makes it difficult for testing irreducibility of  formal power series over integers and over more general domains. This serves as motivation for our present work. So, the aim of this paper is to contribute factorization results and devise new irreducibility criteria for some classes of formal power series over PIDs and some classes of formal power series over discrete valuation domains. In devising our results, we combine the information on factorization properties of the constant term of the underlying power series with the prime factorization of one or more higher order coefficients. Further, we extend the classical Dumas irreducibility criterion \cite{D} to formal power series using Newton polygon arguments, which in particular yields a couple of new  irreducibility criteria for formal power series.

The paper is organized as follows. The main results of this paper are stated and proved in Section \ref{sec:2}. In Section \ref{sec:4}, we provide some explicit examples of formal power series whose factorization properties may be deduced from our results.
\section{Main results and proofs}\label{sec:2}
We recall the prime counting functions $\omega(n)$ and $\Omega(n)$ for a nonzero nonunit $n$ in a UFD $R$, which respectively denote the number of pairwise nonassociate prime divisors of $n$ and the number of pairwise nonassociate prime divisors of $n$ counted with multiplicity. If $R$ is a PID and $f=\sum_{i=0}^\infty a_i z^i \in R[[z]]$, then we define $\Omega_f$ as the number of irreducible factors of $f$ in $R[[z]]$ counted with multiplicity. Using these definitions, we have the following factorization result which provides us an estimate about the number of irreducible factors of a given power series in the UFD $R[[z]]$.
\begin{theorem}\label{th:1}
Let $R$ be a principal ideal domain. If $f=\sum_{i=0}^\infty a_i z^i \in R[[z]]$ is such that $a_0$ is a nonzero nonunit in $R$, then
\begin{eqnarray*}
\omega(a_0)\leq \Omega_f\leq \Omega(a_0).
\end{eqnarray*}
In particular, if $a_0$ is an associate to a prime in $R$, then $f$ is irreducible.
\end{theorem}
\begin{proof}[\bf Proof of Theorem \ref{th:1}]
Let $f(z)=f_1(z)\cdots f_r(z)$ be a product of $r$ irreducible factors $f_1,\ldots,f_r$ in $R[[z]]$ so that $r=\Omega_f$. Then $0\neq a_0=f_1(0)\cdots f_r(0)$. This in view of the fact that $f_i$ is a nonzero nonunit in $R[[z]]$ for each $i$ enforces $r\leq \Omega(a_0)$.

To obtain the remaining inequality, we make the use of Theorem \ref{th:A} as follows.  If $r=1$, then $f$ is irreducible. So, by Theorem \ref{th:A}, we find that $a_0$ is associate to a prime power. Consequently, $\omega(a_0)=1=\Omega_f$ so that the theorem holds for $r=1$. Thus, we may assume that $r>1$. Suppose on the contrary that $r<\omega(a_0)$. Since $a_0=f_1(0)\cdots f_r(0)$ and each $f_i(0)$ is nonzero nonunit in $R$, and since $2\leq r<\omega(a_0)$, it follows that one of the $r$ factors $f_1(0)$, $\ldots$, $f_r(0)$ is not associate to a prime power. Suppose that $f_\alpha(0)$ is not a prime power for some index $\alpha$, where we know that $f_\alpha(0)$ is nonunit in $R$, since $f_\alpha$ is a nonunit in $R[[z]]$. By Theorem \ref{th:A}, $f_\alpha$ is not irreducible in $R[[z]]$. This is a contradiction, since each of $f_1$, $\ldots$, $f_r$ was assumed to be irreducible in $R[[z]]$.
\end{proof}
The lower bound $\omega(a_0)$ in Theorem \ref{th:1} was proved earlier in \cite[Corollary 2.5]{Elliot} but our approach is different. We note that Theorem \ref{th:1} recovers one of the irreducibility criteria of the paper \cite{Daniel2008}. The bounds mentioned in Theorem \ref{th:1} for the number of irreducible factors of $f$ are best possible in the sense that there exist power series attaining these bounds. In view of this, we have the following corollary to Theorem \ref{th:1}.
\begin{corollary}\label{c:1}
Let $R$ be a principle ideal domain. Let $f=\sum_{i=0}^\infty a_i z^i \in R[[z]]$ be such that $a_0$ is nonzero nonunit. If $a_0$ is square free, then $f$ is a product of exactly $\omega(a_0)$ irreducible factors in $R[[z]]$.
\end{corollary}
\begin{proof}[\bf Proof of Corollary \ref{c:1}]
    By Theorem \ref{th:1}, we have $\omega(a_0)\leq \Omega_f\leq \Omega(a_0)$. Since $a_0$ is square free, we must have $\omega(a_0)=\Omega(a_0)$. So, we have $\omega(a_0)=\Omega_f=\Omega(a_0)$.
 \end{proof}
Our next  factorization result provides an upper bound on the number of irreducible factors of a given nonzero nonunit in $\mathbb{Z}[[z]]$ in connection with the prime factorization of the constant and a higher order coefficient.
\begin{theorem}\label{th:2}
    Let $f = \sum_{i=0}^\infty a_i z^i\in \Bbb{Z}[[z]]$ be such that $a_0=\pm p^k$ for some  positive integers $k$ and a prime $p$. Suppose there exists an index $j\geq 1$ for which $p$ does not divide $a_j$. Then $f$ is a product of at most $\min\{k,j\}$ irreducible factors in $\mathbb{Z}[[z]]$. In particular, if $k=1$ or $j=1$, then $f$ is irreducible.
\end{theorem}
The upper bound on the number of irreducible factors of $f$ as mentioned in Theorem \ref{th:2} is best possible as follows from the following example.

For any prime number $p$ and a positive integer $n$, the power series $(p-z)^n\in \mathbb{Z}[[z]]$ is a product of exactly $n$ irreducible factors in $\mathbb{Z}[[z]]$, since $p-z$ is irreducible in $\mathbb{Z}[[z]]$, where $p$ does not divide the coefficient of $z^n$ in $(p-z)^n$ so that here, $j=n$.

Observe that if the polynomial
\begin{eqnarray*}
p^n+p^m \gamma z+a_2z^2+a_3z^3+\cdots+a_d z^d \in \mathbb{Z}[[z]]
\end{eqnarray*}
satisfies the hypothesis of Theorem \ref{th:A2} such that $f$
is reducible in $\mathbb{Z}[[z]]$, then by Theorem \ref{th:2},  $f$ is a product of exactly two irreducible factors in $\mathbb{Z}[[z]]$.

We will now state and prove a more general form of Theorem \ref{th:2} as follows.
\begin{theorem}\label{th:2G}
  Let $(R,v)$ be a discrete valuation domain with a uniformizing parameter $\pi$  for $R$. Let
  $f = \sum_{i=0}^\infty a_i z^i\in R[[z]]$ be such that $a_0= u \pi^k$ for some  unit $u\in R$ and positive integers $k$. Suppose there exists a nonnegative integer $\ell$  and an index $j\geq 1$ such that $\pi^{\ell}$ divides $a_{j}$ but  $\pi^{\ell+1}$ does not divide $a_{j}$. Then $f$ is a product of at most $\min\{k,j+\ell\}$ irreducible factors in $R[[z]]$. In particular, if $k=1$ or $j=1$ and $\ell=0$, then $f$ is irreducible.
\end{theorem}
Theorem \ref{th:2} corresponds to the case $\ell=0$  of Theorem \ref{th:2G} with $R=\mathbb{Z}$, $\pi=p$, and $v=v_p$, the $p$-adic valuation of $\mathbb{Z}$. Further, Theorem \ref{th:B} is precisely the case  $k=j=1=\ell+1$ of Theorem \ref{th:2G}.

In view of Theorem \ref{th:2G}, the power series $f = \sum_{i=0}^\infty a_i z^i\in R[[z]]$ with $a_0=u \pi^k$  is a product of  at most
\begin{eqnarray*}
\min\{k,~v(a_1)+1,~v(a_2)+2,~\ldots\}
\end{eqnarray*}
irreducible factors in $R[[z]]$.
\begin{proof}[\bf Proof of Theorem \ref{th:2G}]
Let $f(z)=f_1(z)\cdots f_r(z)$ be a product of $r$ irreducible factors $f_i$ in $R[[z]]$, $1\leq i\leq r$.  Assume without loss of generality that each $f_i$ is nonconstant. Since $f_i(0)$ is a nonzero nonunit, we have  $v(f_i(0))\geq 1$ for each $i=1,\ldots,r$. Sice $\pi$ is a uniformizing parameter for $R$, we have $v(\pi)=1$ and $f_i(0)=u_i\pi^{k_i}$ for some unit $u_i\in R$ and positive integer $k_i$ for every $i=1,\ldots,r$. Since we have
\begin{eqnarray*}
u\pi^k=f(0)=f_1(0)\cdots f_r(0),
\end{eqnarray*}
it follows that
\begin{eqnarray*}
k=v(u\pi^k)=v(f_1(0))+\cdots +v(f_r(0))=k_1+\cdots+k_r,
\end{eqnarray*}
which in view of the fact that $k_i\geq 1$ for each $i$ establishes $r\leq k$.

In the preceding paragraph, we have shown that $r\leq k$. In view of this, we may assume without loss of generality that $k\geq j+\ell$. Assume on the contrary that $r>j+\ell$, that is, $r-j> \ell$. If we write
\begin{eqnarray*}
f_i=\sum_{t=0}^{\infty}a_{it}z^t\in R[[z]],
\end{eqnarray*}
and substitute for $f_i$ in the equation $f(z)=f_1(z)\cdots f_r(z)$, then we find that
\begin{eqnarray*}
a_j &=& \sum_{i_1+i_2+\cdots+i_r=j}a_{1i_1}a_{2i_2}\cdots a_{ri_r},
\end{eqnarray*}
where the indices under the summation sign satisfy $0\leq i_1,\ldots,i_r\leq j$. Since $r-j>\ell\geq 0$, we have $r>j$. Consequently, for any given $r$-tuple $i_1,\ldots,i_r$ with $0\leq i_1,\ldots,i_r\leq j$ and $i_1+\cdots+i_r=j$, at least $r-j$ of the indices $i_1,\ldots,i_r$ is equal to 0. Since  $f_\alpha(0)=a_{\alpha0}$ and $v(f_\alpha(0))\geq 1$ for each index $\alpha\in \{1,\ldots,r\}$, it follows that  $v(a_{1i_1}a_{2i_2}\cdots a_{ri_r})\geq r-j$ for each such term. We then arrive at the following:
\begin{eqnarray*}
v(a_j)=v\Bigl(\sum_{i_1+i_2+\cdots+i_r=j}a_{1i_1}a_{2i_2}\cdots a_{ri_r}\Bigr)\geq \min_{\sum_{i_1+i_2+\cdots+i_r=j}} \Bigl\{v(a_{1i_1}a_{2i_2}\cdots a_{ri_r})\Bigr\}\geq r-j>\ell
\end{eqnarray*}
which contradicts the hypothesis.
\end{proof}
Our next result is the following irreducibility criterion.
\begin{theorem}\label{th:3}
Let $f=\sum_{n=0}^\infty a_nz^n \in \mathbb{Z}[[z]]$ be such that $a_0=\pm p^k$ for some prime $p$ and positive integer $k\geq 2$. Suppose there exists a positive integer $m$ for which
\begin{eqnarray*}
p^\ell\mid a_{(k-\ell) m+i},~p^{\ell+1}\nmid a_{(k-\ell) m+i},
\end{eqnarray*}
for every $i=1, \hdots, m$, and for every $\ell=1,\hdots,k$. If $p$ does not divide $a_{km+1}$, 
then $f$ is irreducible in $\mathbb{Z}[[z]]$.
\end{theorem}
In particular, Theorem  \ref{th:3} for the case $k=2$ reduces to one of the irreducibility criteria proved in \cite{Daniel2008}. The cases $k=3,4,5$ of Theorem \ref{th:3} were proved recently in the paper \cite{Dutta2024}.

In  \cite{J-S-3}, the authors proved two irreducibility criteria for polynomials over $\mathbb{Z}$. In an attempt to extend  these irreducibility results to the ring of power series, we could obtain the following result.
\begin{theorem}\label{th:4}
    Let $f = \sum_{i=0}^\infty a_i z^i\in \Bbb{Z}[[z]]$ be such that $a_0=\pm p^k$ for some  positive integers $k\geq 2$ and a prime $p$. Suppose there exists an index $j\geq 1$ for which $p^k$ divides $\gcd(a_1,\ldots,a_{j-1})$ and $p$ does not divide $a_j$. If $\gcd(k,j)=1$, then $f$ is  irreducible in $\mathbb{Z}[[z]]$.
\end{theorem}
The following analogue of Eisenstein's irreducibility criterion for power series is immediate from  Theorem \ref{th:4}.
\begin{corollary}\label{c:2}
Let $f=\sum_{i=0}^\infty a_iz^i\in \mathbb{Z}[[z]]$ be such that $a_0=\pm p^k$ for a prime $p$ and positive integer $k$. Suppose there exists a positive integer $n$ coprime to $k$ for which $p^k$ divides $a_i$ for each $i=1,\ldots, {n-1}$, and $p$ does not divide $a_n$. Then $f$ is irreducible in $\mathbb{Z}[[z]]$.
\end{corollary}
In  the sequel, we will prove a generalization of Theorems \ref{th:3} \& \ref{th:4} in a somewhat more general set up (See Theorem \ref{th:6}). To do so, we will need Newton polygon technique for power series. Here, we describe in brief, one such technique developed by Hoffmann \cite{Hoffman1994} for formal power series over a field $\K$ with a valuation $v:\K\rightarrow \mathbb{R}\cup \{\infty\}$ for which $\mathbb{Z}\subseteq v(\K)$. We proceed as follows.
\begin{definition}[Hoffmann \cite{Hoffman1994}]
For any integer $n$, let $f=\sum_{i=n}^\infty a_i z^i \in \K((z))$ be a Laurent series. For each $i$ with $a_i\neq 0$, let
$L_i=\{(i,y)\in \mathbb{R}^2~|~y\geq v(a_i)\}$.
If $a_i=0$, we define $L_i=\emptyset$. Let $L$ be the convex hull of the set $\cup_{i=n}^\infty L_i$. The boundary $\partial L$ of $L$ will be called the Newton polygon $NP(f)$ of $f$.
\end{definition}
\begin{definition}[Hoffmann \cite{Hoffman1994}] Let $\bar{\mathbb{R}}=\mathbb{R}\cup \{\pm \infty\}$ and $\bar{\mathbb{N}}=\mathbb{N}\cup \{\pm \infty\}$. For  $f\in \K[[z]]$,  let $f^*:\bar{\mathbb{R}}\rightarrow \bar{\mathbb{N}}$ be such that if $L$ is a segment of $NP(f)$ with finite slope $s$, and if $\ell$ is the length of the projection of $L$ along $x$-axis, then $f^*(s)=\ell$. If $s\in \mathbb{R}$  is not taken up as a slope of any segment of $NP(f)$, then we set $f^*(s)=0$. Further, we take $f^*(-\infty)=0$ unless $NP(f)$ consists of only $y$-axis in which case, we take $f^*(-\infty)=\infty$. Finally, we define $f^*(\infty)=0$ unless $f$ is a polynomial in which case we take $f^*(\infty)=\infty$.
\end{definition}
We observe that if $f$ factors as $f=gh$ in $\mathbb{Z}[[z]]$, then a segment in $NP(f)$ of slope $s$ can be recovered by translation of segments in $NP(g)$ and $NP(h)$ if and only if
\begin{eqnarray*}
f^*(s)=(gh)^*(s)=g^*(s)+h^*(s).
\end{eqnarray*}
Dumas Theorem \cite{D} in general fails for Newton polygons for power series. This can be seen from the following example:

If we take for a prime $p$, $v=v_p$, the $p$-adic valuation of $\mathbb{Q}$,  and if we take
\begin{eqnarray*}
g=p-pz,~h=p+\sum_{i=1}^\infty p(1-p)z^i\in \mathbb{Z}[[z]],
\end{eqnarray*}
then  $gh=p^2-p^3z$, and we have $g^*(0)=1$, $h^*(0)=\infty$, but $(gh)^*(0)=0$.

However, we have the following result of Hoffmann (see \cite[Theorem 1]{Hoffman1994}), which to some extent extends Dumas Theorem to power series, and this will serve our purpose.
 \begin{thmx}[Hoffmann \cite{Hoffman1994}]\label{th:E} Let $\K$ be a field with a valuation $v$ for which $\mathbb{Z}\subseteq v(\K)$.  Let $f,g\in \K[[z]]$ and suppose that $S=S(f)\leq S(g)$, where $S(f)=\sup\{s\in \bar{\mathbb{R}}~|~f^*(s)\neq 0\}$,  with a similar meaning for $S(g)$. Then each of the following statements holds.
 \begin{enumerate}
    \item[(i)] $(fg)^*(t)=f^*(t)+g^*(t)$ for all $t<S$.
    \item[(ii)] If $f^*(t)+g^*(t)\neq 0$ for infinitely many $t<S$, then  $(fg)^*(t)=0$ for all $t\geq S$.
    \item[(iii)]  Suppose that $f^*(t)+g^*(t)\neq 0$ for only finitely many $t<S$ and suppose further that if $S<\infty$ and $f^*(S)\neq 0$ (in which case necessarily $f^*(S)=\infty$) then $g^*(S)=0$. Then $(fg)^*(S)=\infty$ and $(fg)^*(t)=0$ for all $t>S$.
\end{enumerate}
 \end{thmx}
Note that any nonzero $f\in \K[[z]]$ is invertible if and only if $f(0)$ is a unit in $\K$, that is, $f(0)\neq 0$. Thus, the only irreducible elements in $\K[[z]]$ are associates of $z$.

Now we have the following generalization of Theorems \ref{th:3} \& \ref{th:4} which may be viewed as an extension of Dumas irreducibility criterion \cite{D} to formal power series.
\begin{theorem}\label{th:6}
Let $(R,v)$ be a discrete valuation domain with a uniformizing parameter $\pi$ for $R$. Let $f=\sum_{i=0}^\infty a_iz^i\in R[[z]]$ be such that $a_0=u\pi^k$ for some unit $u\in R$ and a positive integer $k$. Suppose there exists  a positive integer $n$ for which the following conditions are satisfied.
\begin{enumerate}
\item[(i)] $v(a_n)=0$ and $\gcd(k,n)=1$,
\item[(ii)] $\frac{k}{n}<\frac{v(a_i)}{n-i}$ for each $i=1,\ldots,n-1$.
\end{enumerate}
Then  $f$ is irreducible in $R[[z]]$.
\end{theorem}
\begin{proof}[\bf Proof of Theorem \ref{th:6}] Let $\K$ be the field of fraction of $R$. Then $v$ extends to $\K$ with the valuation ring $R$, that is, $R=\{x\in \K~|~v(x)\geq 0\}\cup\{0\}$. We recall that $v(\pi)=1$ and every element of $\K$ and hence of $R$ is an associate to a power of $\pi$. Consequently, $\pi$ generates a prime ideal in $R$, and so, $\pi$ is a prime element in $R$.

Now suppose on the contrary that $f(z)=g(z)h(z)$ for nonunits $g,h\in R[[z]]$. If we let $g=\sum_{i=0}^\infty b_i z^i$ and $h=\sum_{i=0}^\infty c_i z^i$, then we must have $v(b_0)>0$ and $v(c_0)>0$. By the hypothesis, we have
 \begin{eqnarray*}
 \frac{v(a_n)-v(a_0)}{n-0} &=& -\frac{k}{n}> -\frac{v(a_i)}{n-i}= \frac{v(a_n)-v(a_i)}{n-i},~i=1,\ldots,n-1,
 \end{eqnarray*}
 which shows that $NP(f)$ consists of exactly one edge  of negative slope $-k/n$ joining the vertices $(0,k)$ and $(n,0)$, so that $f^*(-k/n)=n$. Since $\pi$ does not divide $a_{n}$, there exist smallest indices $r$ and $s$ for which $v(b_r)=0=v(c_s)$ and $r+s=n$. Let $\alpha$ and $\beta$ be the least indices for which
 \begin{eqnarray*}
 \frac{v(b_\alpha)}{r-\alpha}&=&\min_{0\leq t\leq r-1}\Bigl\{\frac{v(b_t)}{r-t}\Bigr\},~
 \frac{v(c_\beta)}{s-\beta}=\min_{0\leq t\leq s-1}\Bigl\{\frac{v(c_t)}{s-t}\Bigr\},
 \end{eqnarray*}
so that the segment joining the lattice points $(\alpha,v(b_\alpha))$ and $(r,0)$ is the rightmost edge of $NP(g)$ of negative slope, and so, we have $S=S(g)\geq 0$ and $0\leq g^*(t)\leq r$ for all $t$ with $-\infty\leq t<0$. Similarly, the segment joining the lattice points $(\beta,v(c_\beta))$ and $(s,0)$ is the rightmost edge of $NP(h)$ of negative slope  so that $S(h)\geq 0$ and $0\leq h^*(t)\leq s$ for all $t$ with $-\infty \leq t<0$.  Assume that $S(g)\leq S(h)$. Then we have
 \begin{eqnarray*}
 0\leq g^*(t)+h^*(t)\leq r+s,~\text{for}~-\infty\leq t<0.
 \end{eqnarray*}
 By Theorem \ref{th:E} (i), we have
 \begin{eqnarray*}
 f^*(t)=(gh)^*(t)=g^*(t)+h^*(t),~\text{for}~-\infty \leq t<S.
 \end{eqnarray*}
In particular, for $\eta=-k/n$, we have
\begin{eqnarray*}
n=f^*(\eta)=(gh)^*(\eta)=g^*(\eta)+h^*(\eta),
 \end{eqnarray*}
which shows that $g^*(\eta)+h^*(\eta)=n=r+s$. This in view of the fact that $g^*(\eta)\leq r$ and $h^*(\eta)\leq s$ enforces $g^*(\eta)=r$ and $h^*(\eta)=s$. This is possible if and only if
$\alpha=0=\beta$, since $r+s=n$. Consequently, $v(b_\alpha)=v(b_0)>0$ and $v(c_\beta)=v(c_0)>0$, which shows that $v(b_0)<k$ and $v(c_0)<k$. Let $v(b_0)=\ell$ so that $v(c_0)=k-\ell$. We then arrive at the following:
 \begin{eqnarray*}
\eta= -\frac{k}{n} &=& -\frac{\ell}{r}=-\frac{k-\ell}{n-r},
 \end{eqnarray*}
 which is impossible, since $\ell<k$ and $\gcd(k,n)=1$.
\end{proof}
Taking $R=\mathbb{Z}$ and $v=v_p$, the $p$-adic valuation of $\mathbb{Z}$ in Theorem \ref{th:6}, we now proceed to prove Theorems \ref{th:3} \& \ref{th:4} as follows:

To prove Theorem \ref{th:3}, we observe by the hypothesis that $NP(f)$ has exactly one edge of negative slope $-k/(km+1)$ joining the vertices $(0,k)$ and $(km+1,0)$ where we note that $\gcd(k,km+1)=1$. By Theorem \ref{th:6}, $f$ is irreducible in $\mathbb{Z}[[z]]$.

To prove Theorem \ref{th:4}, we observe by the hypothesis that $NP(f)$ has exactly one edge of negative slope $-k/j$ joining the vertices $(0,k)$ and $(j,0)$ where it is given that $\gcd(k,j)=1$. By Theorem \ref{th:6}, $f$ is irreducible.
\begin{remark}
 The hypothesis (ii) in Theorem \ref{th:6} implies that
\begin{eqnarray*}
v(a_i)\geq \Big\lceil \frac{k}{n}(n-i) \Big\rceil.
\end{eqnarray*}
 If we take $n=km+1$ for a positive integer $m$, then in order that the aforementioned condition is satisfied by $f=\sum_{i=0}^\infty a_i z^i\in R[[z]]$ for which $a_0=u\pi^k$, one must have
\begin{eqnarray*}
v(a_{(k-\ell)m+i})\geq \Big\lceil \frac{k}{km+1}(km+1-((k-\ell)m+i)) \Big\rceil=\Big\lceil \ell-\frac{i}{m}+\frac{k-\ell+\frac{i}{m}}{km+1} \Big \rceil=\ell,
\end{eqnarray*}
for all $\ell$ and $i$ satisfying $1\leq \ell\leq k$ and $1\leq i\leq m$. Thus the condition
\begin{eqnarray*}
v(a_{(k-\ell)m+i})\geq \ell
\end{eqnarray*}
in Theorem \ref{th:4}  is the best possible lower bound on respective coefficients in accordance with Theorem \ref{th:6} to deduce irreducibility of $f$. This observation may be used to explain the calculations made in \cite{Dutta2024} for obtaining the irreducibility criteria for  $f$ over integers for the cases $k=3,4,5$.
\end{remark}
\begin{lemma}\label{L:1}
 Let $R$ be a principal ideal domain and $\K$, the field of fraction of $R$. Let  $f=\sum_{i=0}^\infty a_i z^i \in R[[z]]$ be such that $a_0$ is an associate to a prime power $p^k$ in $R$, where $k$ is a positive integer. Let $\mathfrak{p}$ be the prime ideal generated by $p$ in $R$, and let $R_{\mathfrak{p}}$ be the localization of $R$ by $\mathfrak{p}$. If $f$ is irreducible in $R_{\mathfrak{p}}[[z]]$, then $f$ is irreducible in $R[[z]]$.
\end{lemma}
\begin{proof} Since $a_0$ is associate to $p^k$, it follows that $f$ is a nonzero nonunit in $R_{\mathfrak{p}}[[z]]$. Let $v$ be the discrete valuation on $R_{\mathfrak{p}}$ with the uniformizing parameter $p$ for $R_{\mathfrak{p}}$, which extends to $\K$ with the valuation ring $R_{\mathfrak{p}}$. We will prove the contrapositive statement to the conclusion of the Lemma \ref{L:1}. So, let $f$ factors as $f(z)=g(z)h(z)$ for nonzero nonunits $g,h\in R[[z]]$. Then $up^k=f(0)=g(0)h(0)$, which in view of the fact that each of $g$ and $h$ is a nonunit in $R[[z]]$ tells us that $g(0)=u_1p^{k_1}$ and $h(0)=u_2p^{k_2}$, where $u_1$ and $u_2$ are units in $R$ with $u_1u_2=u$ and $k_1$ and $k_2$ are positive integers with $k_1+k_2=k$. This proves that each of $g$ and $h$ is a nonzero nonunit in $R_{\mathfrak{p}}[[z]]$, and so, $f$ factors in $R_{\mathfrak{p}}[[z]]$.
\end{proof}
As an application of Theorem \ref{th:6} and Lemma \ref{L:1}, we have the following factorization result.
\begin{theorem}\label{th:8}
Let $R$ be a principal ideal domain. Let $f=\sum_{i=0}^\infty a_iz^i\in R[[z]]$ be such that $a_0=up_1^{k_1}\cdots p_r^{k_r}$, $r\geq 2$, where $u$ is a unit in $R$, $p_1$, $\ldots$, $p_r$ are $r$ pairwise nonassociate primes in $R$ and $k_1,\ldots,k_r$ are all positive integers. Let $R_{\mathfrak{p}_i}$ be the localization of $R$ by the prime ideal $\mathfrak{p}_i$ generated by the prime $p_i$ with discrete valuation $v_i$ with $p_i$ as a uniformizing parameter for $R_{\mathfrak{p}_i}$ for each $i=1,\ldots r$. Each $v_i$ extends to the field of fraction $\K$ of $R$ with valuation ring $R_{\mathfrak{p}_i}$. Suppose that for each $i=1,\ldots, r$, there exists  a positive integer $n_i$ for which the following conditions are satisfied.
\begin{enumerate}
\item[(i)] $v_{i}(a_{n_i})=0$ and $\gcd(k_i,n_i)=1$,
\item[(ii)] $\frac{k_i}{n_i}<\frac{v_{i}(a_t)}{n_i-t}$ for each $t=1,\ldots,n_i-1$.
\end{enumerate}
Then $f$ is a product of exactly $r$ irreducible factors in $R[[z]]$.
\end{theorem}
\begin{proof}[\bf Proof of Theorem \ref{th:8}] By following the proof of Theorem \ref{th:A}, we find that $f$ factors as $f(z)=f_1(z)\cdots f_r(z)$, where $f_i=\sum_{j=0}^\infty b_{ij}z^j \in R[[z]]$ with $b_{i0}=u_ip_i^{k_i}$ for some unit $u_i$ in $R$ for each $i=1,\ldots, r$ so that $u_1\cdots u_r=u$. In view of this,  it will be sufficient to prove that each $f_i$ is irreducible in $R[[z]]$.

By the hypothesis, we observe that $(R_{\mathfrak{p}_i}, v_i)$ is a discrete valuation domain with uniformizing parameter $p_i$ for $R_{\mathfrak{p}_i}$ and $f_i$ is a nonzero nonunit in $R_{\mathfrak{p}_i}[[z]]$ for each $i$. We will prove that $f_i$ is irreducible in $R_{\mathfrak{p}_i}[[z]]$ for each $i=1,\ldots,r$. We proceed as follows.

For  each $i\in \{1,\ldots,r\}$, let $f^{*_i}$ denotes $f^*$  for $NP(f,p_i)$ with a similar meaning for $f_i^{*_j}$ to be $f_i^*$  for $NP(f_i,p_j)$. We observe that for all $i,j\in \{1,\ldots,r\}$, if $i\neq j$, then $f_i^{*_j}(s)=0$ for all $s<0$. So, in particular, taking $s=-k_i/n_i$ for $i\in \{1,\ldots,r\}$, we have by Theorem \ref{th:E} that
\begin{eqnarray*}
n_i=f^{*_i}(-k_i/n_i) &=& f_1^{*_i}(-k_i/n_i)+\cdots+f_r^{*_i}(-k_i/n_i)= f_i^{*_i}(-k_i/n_i),
\end{eqnarray*}
which proves that $NP(f_i,p_i)$ has a segment of slope $-k_i/n_i$ joining the vertices $(0,k_i)$ and $(n_i,0)$, that is, $v(b_{in_i})=0$ and $\frac{k_i}{n_i}<\frac{v_{p_i}(b_{it})}{n_i-t}$  for each $t=1,\ldots,n_i-1$. Since $\gcd(k_i,n_i)=1$, by Theorem \ref{th:6}, $f_i$ is irreducible in $R_{\mathfrak{p}_i}[[z]]$ for each $i=1,\ldots,r$. By Lemma \ref{L:1}, $f_i$ is irreducible in $R[[z]]$ for every $i=1,\ldots,r$.
\end{proof}
The following two factorization results are immediate from Theorem \ref{th:8}.
\begin{theorem}\label{th:7}
Let $R$ be a principal ideal domain. Let $f=\sum_{i=0}^\infty a_i z^i \in R[[z]]$ be such that $a_0=up_1^{k_1}\cdots p_r^{k_r}$, $r\geq 2$, where $u$ is a unit in $R$; $p_1$, $\ldots$, $p_r$ are $r$ pairwise nonassociate primes in $R$ and $k_1,\ldots,k_r$ are all positive integers. Suppose that for each $\ell=1,\ldots,r$, there exists a  positive integer $n_\ell$ for which the following hold.
\begin{eqnarray*}
p_\ell^{k_\ell}\mid a_i,~i=1,\ldots,n_\ell-1;~p_\ell\nmid a_{n_\ell};~\gcd(k_\ell,n_\ell)=1.
\end{eqnarray*}
Then $f$ is a product of exactly $r$ irreducible factors in $R[[z]]$.
\end{theorem}
\begin{theorem}\label{th:9}
    Let $R$ be a principal ideal domain. Let $f=\sum_{i=0}^\infty a_i z^i \in R[[z]]$ be such that $a_0=up_1^{k_1}\cdots p_r^{k_r}$, $r\geq 2$, where $u$ is a unit in $R$; $p_1$, $\ldots$, $p_r$ are $r$ pairwise nonassociate primes in $R$ and $k_1,\ldots,k_r$ are all positive integers. Suppose that for each $j=1,\ldots,r$, there exists a least positive integer  $m_j$ for which the following hold.
\begin{eqnarray*}
p_j^{\ell}\mid a_{(k_j-\ell)m_j+i},~p_j^{\ell+1}\nmid a_{(k_j-\ell)m_j+i},~p_j\nmid a_{k_jm_j+1},
\end{eqnarray*}
for all $i=1,\ldots, m_j$ and $\ell=1,\ldots, k_j$. Then   $f$ is a product of exactly $r$ irreducible factors in $R[[z]]$.
\end{theorem}
\section{Examples}\label{sec:4}
In this section, we will exhibit some explicit examples of formal power series whose factorization properties may be deduced from the results proved in the preceding Section.
\begin{example}Let $p$ be a prime number. By Theorem \ref{th:1}, the linear polynomial $p+z$ is irreducible in $\mathbb{Z}[[z]]$.  Therefore, for any positive integer $n\geq 2$, the formal power series
$f_1=(p+z)^n$ is a product of exactly $n$ irreducible factors in $\mathbb{Z}[[z]]$. Since $f_1(0)=p^n$, by Theorem \ref{th:1}  we have $1\leq \Omega_{f_5}\leq n$. So, the polynomial $f_1$ attains upper bound mentioned in Theorem \ref{th:1}.
\end{example}
\begin{example}Let $g\in \mathbb{Z}[[z]]$. For a prime number $p$ and positive integers $k$ and $j$ with $j<k$, the  power series
 \begin{eqnarray*}
 f_2   &=& p^k\pm z^j+z^{j+1}g(z) \in \mathbb{Z}[[z]]
 \end{eqnarray*}
satisfies the hypothesis of Theorem \ref{th:2}. Therefore, $f_2$ is a product of at most $\min\{k,j\}=j$ irreducible factors in $\mathbb{Z}[[z]]$. If $j=1$, then $f_2$ is irreducible in $\mathbb{Z}[[z]]$.
\end{example}
\begin{example} For a prime $p$, the localization $\mathbb{Z}_{(p)}$ of $\mathbb{Z}$ by the prime ideal $(p)$ in the ring $\mathbb{Z}$ is a discrete valuation domain with the discrete valuation $v$ with $v(p)=1$. Now in view of Theorem \ref{th:2G}, for any $g\in \mathbb{Z}_{(p)}[[z]]$ and a positive integer $k$, the power series
    \begin{eqnarray*}
f_3 &=& \frac{p^k}{p+1}+pz+pz^2+\cdots +pz^{k-1}+(p+1)z^k +z^{k+1}g(z)\in \mathbb{Z}_{(p)}[[z]]
\end{eqnarray*}
is a  product of at most $\min \{k,2,3,4,\ldots,k\}=2$ irreducible factors in $\mathbb{Z}_{(p)}[[z]]$.
\end{example}
\begin{example}For a  positive integer $k\geq 2$, the power series
\begin{eqnarray*}
f_4 &=& \pm p^{k}+(p^{k}z+p^{k-1}z^2+\cdots+pz^{k})+z^{k+1}(1+z+z^2+\cdots)\in \mathbb{Z}[[z]],
\end{eqnarray*}
satisfies the hypothesis of Theorem \ref{th:3} for $m=1$, and so, $f_4$ is irreducible in $\mathbb{Z}[[z]]$.
\end{example}
\begin{example}
For positive integers $k$ and $j$ with $\gcd(k,j)=1$, the power series
\begin{eqnarray*}
f_5 &=&  (\pm p^{k}+z^j)(1+z+\cdots+z^{j-1}+\cdots)\in  \mathbb{Z}[[z]].
\end{eqnarray*}
can be expressed as $f_5=\pm p^{k}(1+z+\cdots+z^{j-1}+\cdots)+z^j(1+z+\cdots+z^{j-1}+\cdots)$. This
shows that $f_5$ satisfies the hypothesis of Theorem \ref{th:4}. So, $f_5$ is irreducible in   $\mathbb{Z}[[z]]$. Since $(1+z+\cdots+z^{j-1}+\cdots)$ is a unit in $\mathbb{Z}[[z]]$, the polynomial $\pm p^k+z^j\in \mathbb{Z}[[z]]$ being an associate of $f_5$ must also be irreducible in $\mathbb{Z}[[z]]$.
\end{example}
\begin{example} For any field $\K$, consider the polynomial ring $\K[y]$ with the degree valuation $v_\infty$, that is, $v_\infty(0)=\infty$, and for any nonzero polynomial $a\in \K[y]$, $v_\infty(a)=-\deg a$. Then $(\K[y],v_\infty)$ is a discrete valuation domain with any linear polynomial in $\K[y]$ as a uniformizing parameter for $\K[y]$.  The valuation $v_\infty$  extends to the field of fraction $\K(y)$ of $\K[y]$ by defining $v_\infty(a/b)=\deg b-\deg a$ for all $a,b\in \K[y]$ with $b\neq 0$. Now consider the power series $f_6=\sum_{i=0}^\infty a_i(y)z^i\in (\K[y])[[z]]$ for which $a_0=(1+y)^ku$ for nonzero $u\in \K$ and positive integer $k$. Suppose there exists a natural  number $n$ such that
\begin{enumerate}
\item[(i)] $a_n\in \K$ and $\gcd(k,n)=1$,
\item[(ii)] $\frac{k}{n}>\frac{\deg(a_i)}{n-i}$ for each $i=1,\ldots,n-1$.
\end{enumerate}
Then by Theorem \ref{th:6}, the power series $f_6$ is irreducible in $(\K[y])[[z]]$. In particular, if we take $\K=\mathbb{Q}$, we see that the bivariate  polynomial
\begin{eqnarray*}
(1+y)^8+(1+y)^4z+(1+y)^2z^2+y z^3\in  (\mathbb{Q}[y])[[z]]
\end{eqnarray*}
is irreducible.

Moreover, one can obtain factorization results for the formal power series over the polynomial ring $\K[y]$ parallel to Theorems \ref{th:6}-\ref{th:9} by taking $R=\K[y]$ and $v=v_\infty$.
\end{example}
\begin{example} We recall that the ring of Gaussian integers $\Bbb{Z}[\iota]$ is a PID, and so, the power series ring  $(\Bbb{Z}[\iota])[[z]]$ is a UFD. The set of units of $\Bbb{Z}[\iota]$ is $\{\pm1,\pm \iota\}$. An element $a+\iota b \in \Bbb{Z}[\iota]$ is prime if and only if either $ab=0$ and one of $a$ and $b$ is a prime number of the form $4m+3$, or $ab\neq 0$ and $a^2+b^2$ is a prime number  satisfying $a^2+b^2\not\equiv 0,3 \mod 4$. Now for any $g\in  \Bbb{Z}[\iota][[z]]$ and any positive integer $k$, we consider the power series
\begin{eqnarray*}
f_{7}=19^k(4+3\iota)u+4z^{k-1}+z^k g(z)\in (\Bbb{Z}[\iota])[[z]],
\end{eqnarray*}
where $u$ is a unit in $\Bbb{Z}[\iota]$. Note that $19$ and $4+3\iota$  are nonassociate gaussian prime.
The power series $f_7$ satisfies the hypothesis of Theorem \ref{th:7} with $a_0=19^k(4+3\iota)u$, $r=2$, $k_1=k$, $k_2=1$, $n_1=k-1=n_2$  so that $\gcd(k_1,n_1)=1=\gcd(k_2,n_2)$. So, $f_{7}$ is a product of exactly $2$ irreducible factors in $(\Bbb{Z}[\iota])[[z]]$.
\end{example}
\textbf{\large Acknowledgments.}
Senior Research Fellowship (SRF) to Ms. Rishu Garg wide grant no. CSIRAWARD/JRF-NET2022/11769
 from Council of Scientific and Industrial Research (CSIR) is gratefully acknowledged. The present paper is a part of her Ph.D. thesis work.

\subsection*{Disclosure statement}
The authors report that there are no competing interests to declare.

\end{document}